\documentclass[11pt, a4paper]{article}
\usepackage{amsfonts,amsmath,fullpage,amssymb,enumitem,stmaryrd,subfig,fancyhdr,amsthm,hyperref,tikz-cd}

\theoremstyle{definition}
\newtheorem*{mydef}{Definition}

\newcounter{master} \numberwithin{master}{section}
\newtheorem{mythm}[master]{Theorem}
\newtheorem{myprop}[master]{Proposition}
\newtheorem{mycol}[master]{Corollary}
\newtheorem{mylem}[master]{Lemma}

\newcommand{\Aut}{\text{Aut}}
\newcommand{\End}{\text{End}}
\newcommand{\Hom}{\text{Hom}}

\newcommand{\im}{\text{im\;}}
\newcommand{\dlim}{\lim\limits_\to}

\newcommand{\Gl}{\mathfrak{gl}}
\newcommand{\Sl}{\mathfrak{sl}}
\newcommand{\Glm}{\Gl^M}

\newcommand{\soc}{\text{soc}\;}
\newcommand{\Mat}{\text{Mat}}
\newcommand{\Ann}{\text{Ann}}

\newcommand{\hooklongrightarrow}{\lhook\joinrel\longrightarrow}

\newcommand{\Tr}{\text{tr}}
\newcommand{\Id}{\text{Id}}

\begin{document}
\setlength{\baselineskip}{17pt}
\setlist[enumerate]{topsep=0pt,itemsep=-1ex,partopsep=1ex,parsep=1ex}

\begin{titlepage}
\begin{center}
\textsc{\Large Jacobs University Bremen}\\[0.5cm]
\textsc{\Large Thesis for Bachelor of Science in Mathematics}\\[0.5 cm]
\textsc{\Large by Mengyuan Zhang}\\[7 cm]

\textsc{\huge Lie algebras of linear systems\\[0.5 cm] and \\[0.5 cm] their automorphisms}\\[7cm]
\textbf{\large Advisor}\\
\textsc{\large Prof. Dr. Ivan Penkov}\\
\end{center}
\end{titlepage}

\begin{abstract}
\setlength{\baselineskip}{17pt}
The objective of this thesis is to study the automorphism groups of the Lie algebras attached to linear systems. A linear system is a pair of vector spaces $(U,W)$ with a nondegenerate pairing $\langle\cdot,\cdot\rangle\colon U\otimes W\to \mathbb{C}$, to which we attach three Lie algebras $\Sl_{U,W}\subset \Gl_{U,W}\subset\Glm_{U,W}$. If both $U$ and $W$ are countable dimensional, then, up to isomorphism, there is a unique linear system $(V,V_*)$. In this case $\Sl_{V,V_*}$ and $\Gl_{V,V_*}$ are the well-known Lie algebras $\Sl_\infty$ and $\Gl_\infty$, while the Lie algebra $\Glm_{V,V_*}$ is the Mackey Lie algebra introduced in \cite{PSer}.

 We review results about the monoidal categories $\mathbb{T}_{\Sl_{U,W}}$ and $\mathbb{T}_{\Glm_{U,W}}$ of tensor modules, both of which turn out to be equivalent as monoidal categories to the category $\mathbb{T}_{\Sl_\infty}$ introduced earlier in \cite{DPS}. Using the relations between the categories $\mathbb{T}_{\Sl_\infty}$ and $\mathbb{T}_{\Glm_\infty}$, we compute the automorphism group of $\Glm_\infty$. \\[8cm]

\tableofcontents
\end{abstract}


\newpage
\section{Preliminaries}
The ground field is $\mathbb{C}$. $(\;\cdot\;)^*$ denotes the contravariant functor which maps a vector space $M$ to $\Hom(M,\mathbb{C})$ and maps a linear map $f\colon M\to N$ to its dual $f^*\colon  N^*\to M^*$. In this paper, countable dimensional means infinite countable dimensional.  $T(X)$ denotes the tensor algebra of a vector space $X$. $\mathfrak{g}$ denotes an arbitrary Lie algebra, and $\mathfrak{g}$-mod denotes the category of $\mathfrak{g}$-modules, where morphisms are $\mathfrak{g}$-homomorphisms.

\bigskip
Throughout this paper, let $V$ be a fixed vector space with a fixed countable basis $\mathcal{B} = \{e_i\}_{i\in\mathbb{N}}$ and let $e^i$  be the linear functional dual to $e_i$, i.e. $e^i(e_j) = \delta_{ij}\;\;\forall i,j$. Let $\mathcal{B}^* := \{e^i\}_{i\in\mathbb{N}}$ and  $V_* := $ span $\mathcal{B}^*$. We also set $\mathcal{B}_n := \{e_i\}_{i =1}^n$ and $\mathcal{B}_n^* := \{e^i\}_{i =1}^n$. Then clearly $V = \dlim V_n$ and $V_* = \dlim V_n^*$, where $V_n := $ span $\mathcal{B}_n$ and $V_n^* := $ span $\mathcal{B}_n^*$. 

$\Mat_n$ denotes the Lie algebra of $n\times n$ matrices. For $n\le m$, consider the embedding
$\Mat_n\stackrel{\iota_{nm}}{\hooklongrightarrow} \Mat_m$
by upper left corner identity inclusion and filling zeros elsewhere. Then $(\{\Mat_n\}, \{\iota_{nm}\})$ forms a direct system of Lie algebras, and $\Mat_\mathbb{N} := \dlim \Mat_n$ is its direct limit. $\Mat_\mathbb{N}$ is nothing but the Lie algebra of matrices $(A_{ij})_{i,j\in\mathbb{N}}$ that have finitely many nonzero entries.

We define the Lie algebra $\Gl_n:= V_n\otimes V_n^*$, with Lie bracket given by
\begin{equation}\label{one}
[e_i\otimes e^j,e_k\otimes e^l] = \delta_{jk} e_i\otimes e^l-\delta_{il} e_k\otimes e^j\quad \forall e_i,e_k\in \mathcal{B}_n, e^j,e^l\in\mathcal{B}_n^*.
\end{equation}

Observe that $\Gl_n$ and $\Mat_n$ are isomorphic as Lie algebras. An isomorphism is given by sending $e_i\otimes e^j$ to the matrix $E_{ij}$ having 1 at $(i,j)$ and 0 elsewhere. For any $n\le m$, the embedding $\iota_{nm}$ induces an imbedding $\Gl_n\stackrel{\iota'_{nm}}{\hooklongrightarrow} \Gl_m$ such that the following diagram commutes
\[\begin{tikzcd}
  \Gl_n \arrow[hookrightarrow]{r}{\iota'_{nm}}\arrow[leftrightarrow]{d}
  & \Gl_m \arrow[leftrightarrow]{d}\\
  \Mat_n\arrow[hookrightarrow]{r}{\iota_{nm}}
  & \Mat_m, 
\end{tikzcd}\]
making $(\{\Gl_n\}, \{\iota'_{nm}\})$ a direct system with direct limit
\[\Gl_\infty := V\otimes V_* = \dlim (V_n\otimes V_n^*) = \dlim \Gl_n.\]
The Lie bracket on $\Gl_\infty$ is given by (\ref{one}), where now $e_i,e_k\in\mathcal{B},e^j,e^l\in \mathcal{B}^*$. 

Clearly, since the two direct systems $(\{\Mat_n\}, \{\iota_{nm}\})$ and $(\{\Gl_n\}, \{\iota'_{nm}\})$ are isomorphic, their direct limits $\Mat_\mathbb{N}$ and $\Gl_\infty$ are isomorphic. 

The trace homomorphism $\Tr_n\colon \Gl_n\to\mathbb{C}$ can be extended to a homomorphism $\Tr\colon\Gl_\infty\to\mathbb{C}$. Define $\Sl_\infty:= \ker \Tr$. Then $\Sl_\infty = \dlim \ker \Tr_n$. Similarly, $\Sl_n\cong \Mat_n^0$ and $\Sl_\infty\cong \Mat_\mathbb{N}^0$, where $\Mat_n^0$ denotes the Lie subalgebra of $\Mat_n$ consisting of traceless $n\times n$ matrices and $\Mat_\mathbb{N}^0$ denotes the Lie subalgebra of $\Mat_\mathbb{N}$ consisting of traceless finitary matrices.


\section{The Lie algebras $\Sl_{U,W}$ and $\Gl_{U,W}$}
In this section we generalize the construction of $\Sl_\infty$ and $\Gl_\infty$ to arbitrary linear systems $(U,W)$. A pair of vector spaces of arbitrary dimensions $(U,W)$ is called a \emph{linear system} if they are equipped with a nondegenerate bilinear form $\langle\cdot,\cdot\rangle\colon  U\times W\to \mathbb{C}$. We say that $(U',W')$ is a \emph{subsystem} if $U',W'$ are subspaces of $U,W$ respectively and $\langle\cdot,\cdot\rangle$ is nondegenerate when restricted to $U'\times W'$. 
\begin{myprop}
Let $(U,W)$ be a linear system and let $U_f\subset U$ be a finite-dimensional subspace of $U$. Let $W_f$ be a direct complement to $U_f^\perp$ in $W$, i.e. $W = U_f^\perp\oplus W_f$. Then $(U_f,W_f)$ is a subsystem and $\dim U_f = \dim W_f$. 
\end{myprop}
\begin{proof}
Suppose $v\in V_f$ is such that $\langle v,W_f\rangle = 0$. Then $\langle v, W\rangle = 0$, and we must have $v = 0$. Suppose $w\in W_f$ is such that $\langle V_f, w\rangle = 0$. Then $w\in V_f^\perp\cap W_f$, thus $w = 0$. We conclude that $\langle \cdot, \cdot\rangle$ is nondegenerate on $V_f\times W_f$. Therefore the nondegenerate form $\langle\cdot,\cdot\rangle$ induces an injection $W_f\hookrightarrow (V_f)^*$. This forces $W_f$ to be finite-dimensional. Similarly, there is an injection $V_f\hookrightarrow (W_f)^*$. Therefore $\dim V_f = \dim W_f$.
\end{proof}

  Given any linear system $(U,W)$, we define two Lie algebras $\Gl_{U,W}$ and $\Sl_{U,W}$. Let $\Gl_{U,W}$ equal the vector space $U\otimes W$, with Lie bracket
\[
[u_1\otimes w_1,u_2\otimes w_2] = \langle u_2,w_1\rangle u_1\otimes w_2 - \langle u_1, w_2\rangle u_2\otimes w_1 
\quad \forall u_1,u_2\in U, w_1,w_2\in W.
\]
We define $\Sl_{U,W}$ to be the kernel of $\langle\cdot,\cdot\rangle$. Note that the pair $(V,V_*)$ is a linear system, and $\Sl_{V,V_*} = \Sl_\infty$ and $\Gl_{V,V_*} = \Gl_\infty$. The next four propositions generalize some simple observations concerning $\Sl_\infty$ and $\Gl_\infty$ to arbitrary $\Sl_{U,W}$ and $\Gl_{U,W}$.

\begin{myprop}
There are isomorphisms of Lie algebras $\Sl_{U,W}\cong \Sl_{W,U}$ and $\Gl_{U,W}\cong \Gl_{W,U}$.
\end{myprop}
\begin{proof}
Consider the linear operators $f\colon  \Gl_{U,W}\to \Gl_{W,U}$  and $g\colon \Gl_{W,U}\to \Gl_{U,W}$ such that
\[f(u\otimes w) = -w\otimes u\quad\forall u\otimes w\in \Gl_{U,W},\]
 \[g(w\otimes u) = -u\otimes w\quad\forall u\otimes w\in \Gl_{W,U}.\]
 They are both Lie algebra homomorphisms and are mutually inverse. Also
 $f$ restricts to an isomorphism of Lie algebras $\Sl_{U,W}\cong \Sl_{W,U}$.
\end{proof}

\begin{myprop}
The Lie algebra $\Sl_{U,W}$ is simple.
\end{myprop}
\begin{proof}
The set of finite-dimensional subsystems $(U_f,W_f)$ of $(U,W)$ is partially ordered by inclusion, and any two such subsystems have an upper bound. Thus we obtain a direct system of Lie algebras $\{\Sl_{U_f,W_f}\}$ with direct limit $\Sl_{U,W} = \dlim \Sl_{U_f,W_f}$. Any nontrivial ideal $I$ of $\Sl_{U,W}$ intersects nontrivially with some $\Sl_{U_f,W_f}$. We conclude that $I\supset  \Sl_{U_f,W_f}$ by the simplicity of  $\Sl_{U_f,W_f}$. Similarly, $I\supset \Sl_{U_f',W_f'}$ for any finite-dimensional subsystem $(U_f',W_f')$ containing $(U_f,W_f)$ by the simplicity of $\Sl_{U_f',W_f'}$. This means that $I = \Sl_{U,W}$.
\end{proof}

\begin{myprop}
The Lie algebra $\Sl_{U,W}$ is the commutator subalgebra of $\Gl_{U,W}$.
\end{myprop}
\begin{proof}
It is clear from the definition of the Lie bracket that $[\Gl_{U,W},\Gl_{U,W}]\subset \Sl_{U,W}$. Conversely, observe that $[\Sl_{U,W},\Sl_{U,W}]$ is a nontrivial ideal of $\Sl_{U,W}$, therefore equal to $\Sl_{U,W}$ by the simplicity of $\Sl_{U,W}$. To conclude, 
\[\Sl_{U,W} = [\Sl_{U,W},\Sl_{U,W}] = [\Gl_{U,W},\Gl_{U,W}].\qedhere\]
\end{proof}

\begin{myprop}\text{ }
\begin{enumerate}
\item Let $A$ be any finite-dimensional Lie subalgebra of $\Gl_{U,W}$, then $A\subset \Gl_{U_f,W_f}$ for some finite-dimensional subsystem $(U_f,W_f)$.
\item Let $A$ be any finite-dimensional Lie subalgebra of $\Sl_{U,W}$, then $A\subset \Sl_{U_f,W_f}$ for some finite-dimensional subsystem $(U_f,W_f)$.
\end{enumerate}
\end{myprop}
\begin{proof}\text{ }
\begin{enumerate}
\item
Let $\{u_\alpha\}$ be a basis of $U$ and $\{w_\beta\}$ be a basis of $W$. Then $\{u_\alpha\otimes w_\beta\}$ is a basis of $U\otimes W$. Fix a basis $\{a_1,\cdots, a_n\}$ of $A$, and let
\[a_i = \sum_{k = 1}^{n_i} c_{i,k}\cdot u_{\alpha_{i,k}}\otimes w_{\beta_{i,k}}.\]
Let $(U_f,W_f)$ be a finite-dimensional subsystem such that $U_f \supset $ Span$\{u_{\alpha_{i,k}}\}$ and $W_f \supset $ Span$\{w_{\beta_{i,k}}\}$. Then $A\subset$ Span$\{u_{\alpha_{i,k}}\}\otimes$Span$\{w_{\beta_{i,k}}\}\subset \Gl_{U_f,W_f}$.

\item Given a finite-dimensional subalgebra $A\subset \Sl_{U,W} \subset \Gl_{U,W}$, we have $A\subset \Gl_{U_f,W_f}$ for some finite-dimensional subsystem $(U_f,W_f)$ by statement 1. Then $A\subset \Gl_{U_f,W_f}\cap \ker \langle\cdot,\cdot\rangle = \Sl_{U_f,W_f}$.\qedhere
\end{enumerate}
\end{proof}

\bigskip
  Having defined $\Sl_{U,W}$ and $\Gl_{U,W}$ for a linear system $(U,W)$, it is natural to ask when are $\Gl_{U_1,W_1}$ and $\Gl_{U_2,W_2}$ isomorphic as Lie algebras. A necessary and sufficient condition is given in \cite{PSer}.

\begin{mydef} Two linear systems $(U_1,W_1)$ and $(U_2,W_2)$ are isomorphic iff one of the following holds:
\begin{enumerate}
\item There are vector space isomorphisms $f\colon U_1\to U_2$ and $g\colon W_1\to W_2$ such that $\langle f(u),g(w) \rangle = \langle u,w \rangle$ for all $u\in U_1, w\in W_1$.
\item There are vector space isomorphisms $f\colon U_1\to W_2$ and $g\colon W_1\to U_2$ such that $\langle g(w),f(u)\rangle = \langle u,w \rangle$ for all $u\in U_1, w\in W_1$.
\end{enumerate}
We write $(U_1,W_1)\cong (U_2,W_2)$.
\end{mydef}

\begin{mythm}\cite[Prop~1.1]{PSer}\label{slls} 
The Lie algebras $\Sl_{(U_1,W_1)}$ and $\Sl_{(U_2,W_2)}$ are isomorphic iff the linear systems $(U_1,W_1)$ and $(U_2,W_2)$ are isomorphic.
\end{mythm}

\begin{mycol}\label{glls}
The Lie algebras $\Gl_{(U_1,W_1)}$ and $\Gl_{(U_2,W_2)}$ are isomorphic iff the linear systems $(U_1,W_1)$ and $(U_2,W_2)$ are isomorphic.
\end{mycol}
\begin{proof}
$\Longleftarrow$ If we have a linear system isomorphism $f\colon  U_1\to U_2$ and $g\colon W_1\to W_2$, then 
\[f\otimes g\colon  u\otimes w\mapsto f(u)\otimes g(w)\quad \forall u\otimes w\in\Gl_{U_1,W_1}\]
 induces an isomorphism of Lie algebras.
 If we have isomorphisms $f\colon  U_1\to W_2$ and $g\colon W_1\to U_2$, then  
 \[-g\otimes f\colon  u\otimes w\mapsto -g(w)\otimes f(u)\quad\forall u\otimes w\in \Gl_{U_1,W_1}\]
  induces an isomorphism of Lie algebras.
  
$\Longrightarrow$ Conversely, if $\Gl_{U_1,W_1}\cong \Gl_{U_2,W_2}$ then the commutator subalgebras $\Sl_{U_1,W_1}$ and $\Sl_{U_2,W_2}$ are isomorphic. By Proposition \ref{slls} we conclude that $(U_1,W_1)\cong (U_2,W_2)$.
\end{proof}

  We call a linear system $(U,W)$ \emph{countable dimensional} if both $U$ and $W$ are countable dimensional. In such cases we can say more about $\Gl_{U,W}$ and $\Sl_{U,W}$ because of the following observation in \cite{Mackey}. The next theorem and corollary tell us that, up to isomorphism, there is only one countable dimensional linear system. 
  
\begin{mythm}\cite{Mackey} \label{dualbases}
For any countable dimensional linear system $(U,W)$, there exist dual bases $\{\tilde{u}_i\}_{i\in \mathbb{N}}$ of $U$ and $\{\tilde{w}_i\}_{i\in \mathbb{N}}$ of $W$, where $\langle \tilde{u}_i, \tilde{w}_j\rangle = \delta_{ij}\;\forall i,j$.
\end{mythm}
\begin{proof}
 Start with any basis $\{u_i\}_{i\in \mathbb{N}}$ of $U$ and $\{w_i\}_{i\in \mathbb{N}}$ of $W$ and perform the Gram-Schmidt algorithm.
 
  Step 1:  If $\langle u_1,w_1\rangle \ne 0$ then go to step 2. If not, then by non-degeneracy of the form $\langle\cdot,\cdot\rangle$ there is a $w_j$ such that $\langle u_1,w_j\rangle \ne 0$. Replace $w_1$ by $w_1+w_j$.
  
  Step 2:  Set $\tilde{u}_1 = u_1$. Adjust $w_1$ by a scalar to obtain $\tilde{w}_1$ such that $\langle \tilde{u}_1,\tilde{w}_1\rangle = 1$. Replace $u_2$ by $\tilde{u}_2 = u_2-\langle u_2,\tilde{w}_1\rangle \tilde{u}_1$, whence $\langle \tilde{u}_2,\tilde{w}_1\rangle = 0$. Now adjust $w_2$ to make sure $\langle \tilde{u}_2, w_2\rangle = 1$ and set $\tilde{w}_2 = w_2 - \langle u_1,w_2\rangle \tilde{w}_1$. We have $\langle \tilde{u}_1, \tilde{w}_2\rangle = 0$ and $\langle \tilde{u}_2,\tilde{w}_2\rangle = \langle \tilde{u}_2,w_2\rangle - \langle u_1,w_2\rangle \langle\tilde{u}_2,\tilde{w}_1\rangle = 1$.
  
  Step 3:  Having obtained $\tilde{u}_1,\cdots, \tilde{u}_n$ and $\tilde{w}_1,\cdots, \tilde{w}_n$, we set 
\[
\tilde{u}_{n+1} = u_{n+1} - \sum_{i = 1}^n \langle u_{n+1}, \tilde{w}_i\rangle \tilde{u}_i,
\]
 and apply argument in step 1 to adjust $w_{n+1}$ such that $\langle\tilde{u}_{n+1},w_{n+1}\rangle = 1$. We also set
\[
   \tilde{w}_{n+1} = w_{n+1} - \sum_{i = 1}^n \langle \tilde{u}_{i}, w_{n+1}\rangle \tilde{w}_i.
\]
By induction this will give us the dual bases.
\end{proof}

\begin{mycol}
If $(U,W)$ is countable dimensional, then $(U,W)\cong (V,V_*)$. As a consequence $\Sl_{U,W}\cong \Sl_\infty$ and $\Gl_{U,W}\cong \Gl_\infty$.
\end{mycol}
\begin{proof}
Given a pair of dual bases $\{u_n\}_{n\in\mathbb{N}}$ of $U$ and $\{w_n\}_{n\in\mathbb{N}}$ of $W$, the isomorphisms $f:U\to V$ and $g:W\to V_*$ defined by $f(u_n) = e_n$ and $g(w_n) = e^n$ preserve the form $\langle\cdot,\cdot\rangle$. Therefore the linear systems $(U,W)$ and $(V,V_*)$ are isomorphic. By Theorem \ref{slls} and Corollary \ref{glls}, we conclude that there are isomorphisms of Lie algebras $\Sl_{U,W}\cong \Sl_\infty$ and $\Gl_{U,W}\cong \Gl_\infty$.
\end{proof}

  Theorem \ref{dualbases} clearly holds for finite-dimensional linear systems as well, but fails for uncountable dimensional linear systems. Consider $(X: = W\oplus W^*,X':=W\oplus W^*)$ where $W$ is a countable dimensional vector space and the form $X\times X'\to \mathbb{C}$ is given by
  \[(w+w^*,v+v^*)\mapsto w^*(v)+v^*(w)\quad \forall v,w\in W, v^*,w^*\in W^*.\]
If there were a pair of dual bases, then $W^\perp \subset X'$ would be uncountable dimensional. However, we observe that $W^\perp = W\subset X'$.


\bigskip
\section{The Mackey Lie algebras $\Glm_{U,W}$}
In this section we introduce a canonical Lie algebra $\Glm_{U,W}$ corresponding to the linear system $(U,W)$. We discuss the relations between $\Sl_{U,W},\Gl_{U,W}$ and $\Glm_{U,W}$, and realize $U$ and $W$ as non-isomorphic modules of $\Glm_{U,W}$. 

\begin{mydef} We embed $W\hookrightarrow U^*$ and $U\hookrightarrow W^*$ using the pairing $\langle\cdot,\cdot\rangle$. Set 
\[
\Glm_{U,W}:= \{\varphi\in \End(U)|\;\varphi^*(W)\subset W\},\quad
\Glm_{W,U}:= \{\psi\in \End(W)|\;\psi^*(U)\subset U\}.
\] 
Then $\Glm_{U,W}$ and $\Glm_{W,U}$ are Lie subalgebras of the Lie algebras $\End(U)$ and $\End(W)$, respectively. 

It's easy to see that there is an isomorphism $\Glm_{U,W} \cong \Glm_{W,U}$ of Lie algebras. The map that sends $\varphi \in \Glm_{U,W}$ to $-\varphi^*_{|W}$, is a Lie algebra homomorphism whose inverse homomorphism is given by sending $\psi\in \Glm_{W,U}$ to $-\psi^*_{|U}$.
\end{mydef}

\begin{myprop}[Penkov, Serganova]
The Lie algebra $\Gl_{U,W}$ is isomorphic to the ideal $S$ of $\Glm_{U,W}$, where
\[
S := \{\varphi\in \Glm_{U,W}\colon \dim \varphi(U)<\infty, \dim\varphi^*(W)<\infty\}.
\]
\end{myprop}
\begin{proof} First of all, it is clear that $[h,\varphi] = h\varphi-\varphi h$ has finite-dimensional image in both $U$ and $W$ if one of $h,\varphi$ does. Thus, $S$ is an ideal.
Consider the injection of Lie algebras $\iota\colon \Gl_{U,W}\to \Glm_{U,W}$ induced by $u\otimes w \mapsto \langle \cdot,w\rangle u$. The range of $\iota$ is indeed in $\Glm_{U,W}$ because  
\[(\iota(u\otimes w))^*(\langle\cdot,y\rangle) = \langle \langle \cdot,w\rangle u,y\rangle = \langle  \cdot,\langle u,y\rangle w\rangle\in W\quad\forall \langle\cdot,y\rangle\in W.\]
Moreover, since elements of $U\otimes W$ are finite linear combinations of pure tensors, under $\iota$ they have finite-dimensional images in both $U$ and $W$, which means that $\im \iota\subset S$.

 We now show that $\im \iota \supset S$. Let $(U_f,W_f)$ be a finite-dimensional subsystem of $(U,W)$. Consider the Lie subalgebra 
\[
S_{U_f,W_f} := \{\varphi\in \End(U)\colon  \varphi(U)\subset U_f \text{ and } \varphi^*(W)\subset W_f\}\subset \Glm_{U,W}.
\] 
Observe that $\langle \varphi(v),W\rangle = \langle v,\varphi^*(W)\rangle = 0$ for all $v\in W_f^\perp.$
Therefore $U\supset \ker \varphi \supset W_f^\perp$ for any $\varphi\in S_{U_f,W_f}$.  In addition, since $\varphi(U)\subset U_f$, the subspace $S_{U_f,W_f}$ can at most have dimension $\dim U_f \cdot \dim W_f$. Since $\Gl_{U_f,W_f}$ injects into $S_{U_f,W_f}$ under $\iota$ and does have dimension $\dim U_f\cdot \dim W_f$, we conclude that $\iota$ restricts to an isomorphism from $\Gl_{U_f,W_f}$ to $S_{U_f,W_f}$. Since any $\varphi\in S$ lies in some $S_{U_f,W_f}$, we have $\im\iota = S$.\end{proof}
\begin{mycol}
If $(U,W)$ is a finite-dimensional linear system, then $\Gl_{U,W}\cong \Glm_{U,W}$.
\end{mycol}

  We identify $\Gl_{U_f,W_f}$ with $S_{U_f,W_f}$, and $\Gl_{U,W}$ with $S$ inside $\Glm_{U,W}$ whenever appropriate.

\begin{myprop}\cite[Lemma~6.1]{PSer} 
The Lie algebra $\Glm_{U,W}$ has a unique simple ideal $\Sl_{U,W}$.
\end{myprop}

\begin{myprop}\label{glmls}
The Lie algebras $\Glm_{U_1,W_1}$ and $\Glm_{U_2,W_2}$ are isomorphic iff the linear systems $(U_1,W_1)$ and $(U_2,W_2)$ are isomorphic.
\end{myprop}
\begin{proof}
$\Longleftarrow$ This is obvious because the definition of $\Glm_{U,W}$ is intrinsic to the linear system $(U,W)$. In the following we construct the isomorphism explicitly. If we have a linear system isomorphism $f\colon U_1\to U_2$ and $g\colon W_1\to W_2$, then we first show that $f^*_{|W_2} = g^{-1}$ and $(f^{-1})^*_{|W_1} = g$. Indeed
\[
f^*(\langle\cdot,w_2\rangle) = \langle f(\cdot),w_2\rangle = \langle f^{-1}f(\cdot), g^{-1}(w_2)\rangle = \langle \cdot, g^{-1}(w_2)\rangle
\quad \forall w_2\in W_2,
\]
\[
(f^{-1})^*(\langle\cdot,w_1\rangle) = \langle f^{-1}(\cdot),w_1\rangle = \langle ff^{-1}(\cdot), g(w_1)\rangle = \langle \cdot, g(w_1)\rangle
\quad \forall w_1\in W_1.
\]

Consider the map that sends $\varphi\in \Glm(U_1,W_1)$ to $f \varphi f^{-1}$. By the above
\[
(f\varphi f^{-1})^*_{|W_2} = ((f^{-1})^* \varphi^* f^*)_{|W_2} = g\varphi^* g^{-1},
\]
where the right hand side keeps $W_2$ stable. This map is a homomorphism of Lie algebras between $\Glm_{U_1,W_1}$ and $\Glm_{U_2,W_2}$, which has an inverse that sends $\psi\in \Glm_{U_2,W_2}$ to $f^{-1}\psi f$.

If we have a linear system isomorphism $f\colon U_1\to W_2$ and $g\colon W_1\to U_2$, then the map that sends $\varphi$ to $f\varphi f^{-1}$ is an isomorphism of Lie algebras $\Glm_{U_1,W_1}\cong \Glm_{W_2,U_2}$ by the similar argument as above. But the latter is isomorphic to $\Glm_{U_2,W_2}$.

 $\Longrightarrow$ Conversely if $\Glm_{U_1,W_1}\cong \Glm_{U_2,W_2}$, then $\Sl_{U_1,W_1}\cong\Sl_{U_2,W_2}$ because they are both the unique simple ideals of the respective Lie algebras. By Theorem \ref{slls}, the linear systems $(U_1,W_1)$ and $(U_2,W_2)$ are isomorphic.
  \end{proof}
\begin{mycol}
If $(U,W)$ is countable,  then $\Glm_{U,W}\cong  \Glm_{V,V_*}$.
\end{mycol}

  In addition, we observe the following about $\Glm_{U,W}$.
\begin{myprop}\
\begin{enumerate}
\item $\Glm_{U,W}\subset \Glm_{U,U^*} = \End(U)$.
\item $\Glm_{U,W}$ has one dimensional center $\mathbb{C}$Id.
\item $\Gl_{U,W}\oplus \mathbb{C}$Id is an ideal of $\Glm_{U,W}$.
\end{enumerate}
\end{myprop}
\begin{proof}\text{ }
\begin{enumerate}
\item Obvious.
\item Assume to the contrary that $\varphi$ is in the center of $\Glm_{U,W}$ and $\varphi\not\in \mathbb{C}\Id$. Then there exists $u\in U$ such that $\varphi(u) = u'$ and $u\ne \lambda u'$ for any $\lambda\in\mathbb{C}$, in particular $u\ne 0$. Extend the linearly independent set $\{u,u'\}$ to a basis of $U$, and define a linear operator $\psi$ such that $\psi(u') = u$ and $\psi(v) = 0$ for any $v$ in the basis that does not equal $u'$. Clearly $\psi\in\Gl_{U,W}\subset \Glm_{U,W}$, but 
\[
\varphi\circ\psi(u) = 0 \ne u = \psi\circ\varphi(u).
\]
Contradiction.
\item $\mathbb{C}$Id $\cap \Gl_{U,W} = 0$. The sum of two ideals is an ideal. \qedhere
\end{enumerate}
\end{proof}

Define $\Glm_\infty := \Glm_{V,V_*}$. Then $\Glm_\infty$ is isomorphic to the Lie algebra $\Glm_{U,W}$ for any countable dimensional linear system $(U,W)$. In fact, $\Glm_\infty$ can be conveniently thought of as a certain Lie algebra of matrices.

Define $\Mat^M_{\mathbb{N}}$ to be the vector space consisting of matrices $(A_{ij})_{i,j\in\mathbb{N}}$ such that each row and each column has finitely many nonzero entries. Then $\Mat^M_{\mathbb{N}}$ is a Lie algebra and there is an isomorphism of Lie algebras $\Glm_\infty\cong \Mat^M_{\mathbb{N}}$.

  We obtain the following commutative diagrams
\[\begin{tikzcd}
  \Sl_\infty \arrow[hookrightarrow]{r}\arrow[leftrightarrow]{d}
  & \Gl_\infty \arrow[hookrightarrow]{r}\arrow[leftrightarrow]{d}
  & \Glm_\infty \arrow[leftrightarrow]{d}\\
  \Mat^0_\mathbb{N}\arrow[hookrightarrow]{r}
  & \Mat_\mathbb{N}\arrow[hookrightarrow]{r}
  & \Mat^M_\mathbb{N}. 
\end{tikzcd}\]

\bigskip
  We know that $V$ and $V^*$ have dual $\Gl_\infty$-module structures and $V_*$ is a submodule of $V^*$. Generalizing this, we give $U$ and $U^*$ dual $\Glm_{U,W}$-module structures, and let $W$ be a submodule of $U^*$. By restriction, $U$ and $W$ become $\Gl_{U,W}$ and $\Sl_{U,W}$ modules as well. The actions are explicitly given by
\begin{align*}
& \varphi\cdot u = \varphi(u),\quad \varphi\cdot w = -\varphi^*(w)\quad\text{for } \varphi\in \Glm_{U,W}, u\in U,w\in W,\\
& u\otimes w\cdot x = \langle x,w\rangle u,\quad u\otimes w\cdot y = -\langle u,y\rangle w\quad\text{for } u\otimes w\in \Gl_{U,W}, x\in U, y\in W.
\end{align*}

Note that for any nonnegative integers $p,q$, the tensor product $U^{\otimes p}\otimes W^{\otimes q}$ also becomes a $\Glm_{U,W}$-module. Moreover, $U$ and $W$ are not isomorphic as $\Glm_{U,W},\Gl_{U,W}$ or $\Sl_{U,W}$-modules, simply because $U\not\cong W$ as $\Sl_{U_f,W_f}$-modules for finite-dimensional subsystems $(U_f,W_f)$ with $\dim U_f = \dim W_f \ge 3$. Indeed, we know the following decomposition of $\Sl_{U_f,W_f}$-modules
\[
U\cong U_f\oplus \left(\bigoplus_{n\in\mathbb{N}}\mathbb{C}\right)\quad\text{and}\quad 
W\cong W_f\oplus \left(\bigoplus_{n\in\mathbb{N}}\mathbb{C}\right),
\]
but $W_f\cong U_f^*\not\cong U_f$ as $\Sl_{U_f,W_f}$-modules if $\Sl_{U_f,W_f}\not\cong \Sl(2)$.
\bigskip


\section{The categories $\mathbb{T}_{\Sl_{U,W}}$, $\mathbb{T}_{\Glm_{U,W}}$ and the functor $\mathcal{F}^h$ }
In this section we present results about the categories $\mathbb{T}_{\Sl_{U,W}}$ and $\mathbb{T}_{\Glm_{U,W}}$ introduced in \cite{PSer}. In particular, an object in $\mathbb{T}_{\Sl_{U,W}}$ or $\mathbb{T}_{\Glm_{U,W}}$ is isomorphic to a finite length submodule of a direct sum of finitely many copies of $T(U\oplus W)$.  Moreover, the categories of tensor modules $\mathbb{T}_{\Sl_{U,W}}$ and $\mathbb{T}_{\Glm_{U,W}}$ are both equivalent to $\mathbb{T}_{\Sl_\infty}$ as monoidal categories. 

\begin{mydef} Let $M$ be an $\Sl_{U,W}$-module, we say $M$ is $\Sl_{U,W}$-integrable iff for every $\varphi\in \Sl_{U,W},m\in M$, we have $\dim\{m,\varphi\cdot m,\varphi^2\cdot m,\cdots\}<\infty$.

A subalgebra $\mathfrak{k}$ of $\Sl_{U,W}$ has \emph{finite co-rank} iff it contains $\Sl_{W_f^\perp,U_f^\perp}$ for some finite-dimensional subsystem $(U_f,W_f)$. We say that a $\Sl_{U,W}$-module $M$ satisfies the \emph{large annihilator condition} if for every $m\in M$, its annihilator $\Ann(m):= \{g\in \Sl(U,W)\colon g\cdot m = 0\}$ is a finite co-rank subalgebra of $\Sl_{U,W}$.

 Define $\mathbb{T}_{\Sl_{U,W}}$ to be the full subcategory of $\Sl_{U,W}$-mod where the objects are finite length $\Sl_{U,W}$-integrable modules that satisfy the large annihilator condition. 
\end{mydef}

Despite the abstract definition, the category $\mathbb{T}_{\Sl_{U,W}}$ is nothing but the category of tensor modules.

\begin{mythm}\cite[Corollay~5.12]{PSer} The following are equivalent:
\begin{enumerate}
\item $M$ is an object of $\mathbb{T}_{\Sl_{U,W}}$. 
\item $M$ is isomorphic to a finite length submodule of a direct sum of finitely many copies of $T(U\oplus W)$;
\item $M$ is isomorphic to a finite length subquotient of a direct sum of finitely many copies of $T(U\oplus W)$.
\end{enumerate}
\end{mythm}

\begin{mythm}\cite[Theorem~5.5]{PSer}
For any linear system $(U,W)$, there is an equivalence $\mathbb{T}_{\Sl_{U,W}}\leftrightsquigarrow\mathbb{T}_{\Sl_\infty}$ of monoidal tensor categories.
\end{mythm}

The category $\mathbb{T}_{\Glm_{U,W}}$ is defined analogously, but with slight differences.  A subalgebra $\mathfrak{k}$ of $\Glm_{U,W}$ is called \emph{finite co-rank} iff $\mathfrak{t}\supset \Glm_{W_f^\perp,U_f^\perp}$ for some finite-dimensional subsystem $(U_f,W_f)$. A $\Glm_{U,W}$-module $M$ is said to satisfy the \emph{large annihilator condition} if $\Ann(m)$ is a finite co-rank subalgebra of $\Glm_{U,W}$ for any $m\in M$. 

We define $\mathbb{T}_{\Glm_{U,W}}$ to be the full subcategory of $\Glm_{U,W}$-mod whose objects are finite length $\Sl_{U,W}$-integrable modules that satisfy the large annihilator condition. 

Similar to the $\mathbb{T}_{\Sl_{U,W}}$ case, the category $\mathbb{T}_{\Glm_{U,W}}$ is nothing but the category of tensor modules.

\begin{mythm}\cite[Theorem~7.9~a)]{PSer} The following are equivalent:
\begin{enumerate}
\item $M$ is an object of $\mathbb{T}_{\Glm_{U,W}}$;
\item $M$ is isomorphic to a finite length submodule of a direct sum of finitely many copies of $T(U\oplus W)$;
\item $M$ is isomorphic to a finite length subquotient of a direct sum of finitely many copies of $T(U\oplus W)$.
\end{enumerate}
\end{mythm}

\begin{mythm}\cite[Theorem~7.9~b)]{PSer}
For any object $M$ in the category $\mathbb{T}_{\Glm_{U,W}}$, let Res$(M)$ denote $M$ regarded as a module of $\Sl_{U,W}$. Then Res: $\mathbb{T}_{\Glm_{U,W}}\rightsquigarrow\mathbb{T}_{\Sl_{U,W}}$ is a well defined functor that is fully faithful and essentially surjective. Moreover, Res$(M\otimes N) =$ Res$(M)\otimes$Res$(M)$. Thus $\mathbb{T}_{\Glm_{U,W}}$ and $\mathbb{T}_{\Sl_{U,W}}$ are equivalent as monoidal tensor categories.
\end{mythm}

We introduce the notion of \emph{modules twisted by an automorphism}, or \emph{twisted modules}. Let $M$ be a $\mathfrak{g}$-module, and $h\in\Aut(\mathfrak{g})$ be an automorphism of $\mathfrak{g}$. Then the module twisted by $h$ is the same as $M$ as a vector space but with new action
\[
\varphi\cdot_{new} w = h(\varphi)\cdot_{old} w\;\;\;\;\forall \varphi\in \mathfrak{g}, \forall w\in M.
\] 
We denote this new $\mathfrak{g}$-module structure by $M^h$ . In the language of representations, the pull-back of the representation $\rho\colon  \mathfrak{g}\to\Gl(M)$ along $h\colon\mathfrak{g}\to\mathfrak{g}$ gives the representation $\rho\circ h$, which corresponds to the twisted module $M^h$.

Recall that the \emph{socle} of a module $M$, denoted by $\soc M$, is the sum of simple submodules of $M$. If $f:M\to N$ is a $\mathfrak{g}$-module isomorphism, then clearly $f(\soc M) = \soc N$.

We observe the following about twisted modules.
\begin{myprop}\label{functor}
 Let $U,W$ be $\mathfrak{g}$-modules and $g,h\in\Aut(\mathfrak{g})$. Then the following holds:
\begin{enumerate}
\item $(W^h)^g = W^{g\circ h}$ as $\mathfrak{g}$-modules.
\item If $f:W\to U$ is a $\mathfrak{g}$-module homomorphism, then $f:W^h\to U^h$ is also a $\mathfrak{g}$-module homomorphism. In particular, if there is an isomorphism of $\mathfrak{g}$-modules $W\cong U$, then there is also an isomorphism of $\mathfrak{g}$-modules $W^h\cong U^h$.
\item Any $h\in\Aut(\mathfrak{g})$ induces a covariant functor $\mathcal{F}^h$ that sends $W$ to $W^h$, and a morphism $f:W\to U$ to $f:W^h\to U^h$. The functor $\mathcal{F}^h$ has an inverse functor $\mathcal{F}^{h^{-1}}$, thus it is an automorphism of the category $\mathfrak{g}$-mod.
\item The functor $\mathcal{F}^h$ commutes with the contravariant functor $(\;\cdot\;)^*$, that is, $(U^*)^h = (U^*)^h$. 
\item The functor $\mathcal{F}^h$ preserves the socle of a module, that is, $\soc W^h = (\soc W)^h$.
\end{enumerate}
\end{myprop}
\begin{proof} \text{ }
\begin{enumerate}
\item The $\mathfrak{g}$-modules $(W^h)^g$ and $W^{g\circ h}$ have the same action
\[
\varphi\cdot_{new} w = g(h(\varphi))\cdot_{old} w\;\;\;\;\forall \varphi\in\mathfrak{g},w\in W.
\]
\item Since $f\colon W\to U$ is a $\mathfrak{g}$-homomorphism, for every $\varphi\in \mathfrak{g}$ we have the following commutative diagram
\[
\begin{tikzcd}
W \arrow{r}{f}\arrow{d}[left]{\varphi} 
  & U\arrow{d}{\varphi}\\
W \arrow{r}{f}
  & U.
\end{tikzcd}
\]
Then for every $\varphi\in\mathfrak{g}$ we also have the following commutative diagram
\[\begin{tikzcd}
W^h \arrow{r}{f}\arrow{d}[left]{h(\varphi)} 
  & U^h\arrow{d}{h(\varphi)}\\
W^h \arrow{r}{f}
  & U^h.
\end{tikzcd}\]
Thus $f\colon W^h\to U^h$ is a $\mathfrak{g}$-module homomorphism.
\item It is clear from statement 2 that $\mathcal{F}^h$ takes a module $W$ to $W^h$ and morphisms $f\colon W\to U$ to the same morphism $f\colon W^h\to U^h$, and is a well-defined functor from $\mathfrak{g}$-mod to itself. The functors $\mathcal{F}^{h^{-1}}$ and $\mathcal{F}^{h}$ are mutually inverse by statement 1, thus $\mathcal{F}^h$ is an automorphism of the category $\mathfrak{g}$-mod.
\item Both actions are given by
\[
(\varphi\cdot u^*)(u) = -u^*(h(\varphi)\cdot u) \quad \forall \varphi\in\mathfrak{g},u\in U, u^*\in U^*.
\]
\item Since $\mathcal{F}^{h}$ has an inverse $\mathcal{F}^{h^{-1}}$, it maps simple submodules to simple submodules. Also it is easy to see that for any submodules $W_1,W_2\subset U$, we have $(W_1+W_2)^h = W_1^h+W_2^h$. Therefore $\soc W^h = (\soc W)^h$.
\qedhere
\end{enumerate}
\end{proof}

One important observation is that $U^h$ and $W^h$ are simple $\Sl_{U,W}$-modules for any $h\in\Aut(\Sl_{U,W})$. Further, we have $\soc (U^h)^* = W^h$ and $\soc (W^h)^* = U^h$. 
\begin{myprop}\label{socle}
As $\Sl_{U,W},\Gl_{U,W}$ and $\Glm_{U,W}$-modules, 
\begin{enumerate}
\item $U$ and $W$ are simple.
\item $\soc U^* = W$ and $\soc W^* = U$.
\item $\soc (U^h)^* = W^h$ and $\soc (W^h)^* = U^h$ for any $h\in\Aut(\Glm_\infty)$.
\end{enumerate}
\end{myprop}
\begin{proof} \
\begin{enumerate}
\item It suffices to check the simplicity of $U$ and $W$ as $\Sl_{U,W}$-modules. Let $X\subset U$ be a nontrivial submodule of $U$, then it intersects nontrivially with some finite-dimensional subspace $U_f\subset U$. Since $U_f$ is a simple $\Sl_{U_f,W_f}$-module, we must have $X\supset U_f$. Therefore we must have $X\supset U_f'$ for any finite-dimensional subspace $U_f'\supset U_f$. We conclude that $X = U$, since $U$ is the union of all its finite-dimensional subspaces. The same argument applies to $W$.
\item Let $X$ be a nontrivial $\Sl_{U,W}$-submodule of $U^*$. It suffices to show that $X\supset W$. Since $W$ is a simple module, all we need to show is that $X\cap W\ne 0$. Suppose the contrary, then take any nonzero element $u^*$ in $X$, we must have $u^*\in U^*- W$. There exists some $u\in U$ such that $\langle u,u^*\rangle\ne 0$ by the non-degeneracy of the form $\langle\cdot,\cdot\rangle$. Pick some nonzero $w\in (\mathbb{C}u)^\perp\subset W$, then $u\otimes w\in \Sl_{U,W}$. However $u\otimes w\cdot u^* = \langle u,u^*\rangle w\in W\cap X$. Contradiction. 
The same argument can be applied to prove that $\soc W^* = U$.
\item By Proposition \ref{functor} statement 4 and 5, 
\[
\soc (U^h)^* = \soc (U^*)^h = (\soc U^*)^h = W^h,
\]
\[
\soc (W^h)^* = \soc (W^*)^h = (\soc W^*)^h = U^h. \qedhere
\]
\end{enumerate}
\end{proof}

\bigskip


\section{Automorphisms of $\Sl_\infty$ and $\Glm_\infty$}
In this section we investigate the automorphism groups of $\Sl_\infty$ and $\Glm_\infty$.

We say the linear system $(U,W)$ is \emph{self-dual} if there is an isomorphism of vector spaces $f\colon U\to W$ such that $\langle f^{-1}(w),f(u)\rangle = \langle u,w\rangle$ for every $u\in U,w\in W$. In this case, according to the proof of Proposition \ref{glmls}, we observe that
\[
f^*_{|U} = f,\quad (f^{-1})^*_{|W} = f^{-1}.
\]

By Proposition \ref{glmls}, there is an isomorphism of Lie algebras $\gamma\colon\Glm_{U,W}\to\Glm_{W,U}$ defined by $\gamma\colon\varphi\mapsto f\varphi f^{-1}$. The Lie algebra $\Glm_{W,U}$ is again isomorphic to $\Glm_{U,W}$ by the map $\sigma\colon\varphi\mapsto -\varphi^*_{|U}$. The composition $\tau: =\sigma\gamma$ is an automorphism of $\Glm_{U,W}$, which is explicitly 
\[
\tau\colon \varphi\mapsto - f^{-1}\varphi^* f\quad \forall \varphi\in \Glm_{U,W}.
\]

 From the definition of $\tau$ we see that $\tau$ is an involution. Moreover, for every $\varphi\in \Glm_{U,W}$ the following diagram commutes 
\[\begin{tikzcd}
U \arrow{r}{f}\arrow{d}[left]{\tau(\varphi)} 
  & W\arrow{d}{-\varphi^*}\\
U \arrow{r}{f}
  & W.
\end{tikzcd}\]
Therefore we conclude that $f:U^\tau\to W$ is an isomorphism of $\Glm_{U,W}$-modules. Consequently $W^\tau \cong (U^\tau)^\tau =  U$ as $\Glm_{U,W}$-modules. Since the Lie algebra $\Glm_{U,W}$ has a unique simple ideal $\Sl_{U,W}$, the automorphism $\tau$ restricts to an automorphism of $\Sl_{U,W}$. 

\bigskip
If the linear system $(U,W)$ has a pair of dual bases relative to $\langle\cdot,\cdot\rangle$, then clearly $(U,W)$ is self-dual. The converse is false. Again, consider the linear system $(X,X')$ which is defined at the end of section 2. Let $f:X\to X'$ be the isomorphism given by $f(w,w^*) = (w,w^*)$. Then $\langle f^{-1}(x), f(y)\rangle =  \langle y, x\rangle$ for every $x\in X', y\in X$.  Therefore the linear system $(X,X')$ is self-dual, but it does not have a pair of dual bases. 

In the countable dimensional case, the linear system $(V,V_*)$ has a pair of dual bases. Let $\epsilon:V\to V_*$ be the isomorphism of vector spaces induced by $\epsilon(e_n) = e^n\;\forall e_n\in\mathcal{B}$. Then we observe that $\tau: \varphi\mapsto -\epsilon^{-1}\varphi^* \epsilon$ is the involution of the Lie algebra $\Mat^M_{\mathbb{N}}$ given by $A\mapsto -A^t$.

\bigskip
In the following, let $\mathfrak{g} = \Sl_\infty, \Glm_\infty$. Define $\tilde{G}:= \{g\in\Aut(V)|\;g^*(V_*) = V_*\}$. Conjugation by an element in $\tilde{G}$ induces an automorphism of $\mathfrak{g}$. Therefore we have a group representation $\rho:\tilde{G}\to\Aut(\mathfrak{g})$ with kernel $\mathbb{C}^*$, where $\mathbb{C}^*$ denotes nonzero complex numbers. We set $\tilde{G}_0 := \im \rho$. 

\begin{mythm}[Beidar et.al, \cite{BBCM}]\label{autsl}
The subgroup $\tilde{G}_0$ has index $2$ in $\Aut(\Sl_\infty)$, the quotient is represented by the involution $\tau$.
In fact, $\Aut(\Sl_\infty) = \tilde{G}_0\rtimes \{\Id, \tau\} $.
\end{mythm}

\begin{mycol}\label{slundef}
If $h\in \tilde{G}_0$, then $V^h\cong V$ and $V^{\tau\circ h}\cong V_*$ as $\Sl_\infty$-modules. 
\end{mycol}
\begin{proof}
If $h\in \tilde{G}_0$, then $h(\varphi) = g^{-1}\varphi g$ for every $\varphi\in \Sl_\infty$ where $g\in \tilde{G}$. It follows that $g:V^h\to V$ is an automorphism of $\Sl_\infty$-modules. Therefore $V^{\tau\circ h} = (V^h)^\tau \cong V^\tau\cong V_*
$ as $\Sl_\infty$-modules.
\end{proof}

Let $M$ be a $\Glm_\infty$-module. We say that $\Sl_\infty$ \emph{acts densely} on $M$, if for every $\varphi\in \Glm_\infty$ and any choice of finitely many vectors $r_1,\cdots, r_n\in M$, there exists $\psi\in \Sl_\infty$ such that $\psi \cdot r_k = \varphi \cdot r_k$ for $k = 1,\cdots, n$.

\begin{mylem}\cite[Lemma~8.2]{PSer}\label{slglm}
Let $L$ and $L'$ be $\Glm_\infty$-modules which have finite length as $\Sl_\infty$-modules. Then
\[
\Hom_{\Sl_\infty}(L,L') = \Hom_{\Glm_\infty}(L,L').
\]
In particular, if $L$ and $L'$ are isomorphic as $\Sl_\infty$-modules, then they are isomorphic as $\Glm_\infty$-modules.
\end{mylem}

\begin{myprop}\cite[Proposition~8.1]{PSer}\label{actsdensely}
Let $M$ be a finite length $\Glm_\infty$-module that is $\Sl_\infty$-integrable. Then $M$ is an object of $\mathbb{T}_{\Glm_\infty}$ iff 
$\Sl_\infty$ acts densely on it.
\end{myprop}

\begin{mycol}\label{undef}
Let $h\in\Aut(\Glm_\infty)$. Then there exists an isomorphism of $\Glm_\infty$-modules $V^h\cong V$ or $V^h\cong V_*$.
\end{mycol}
\begin{proof} 
If $h\in\Aut(\Glm_\infty)$, then $h$ restricts to an automorphism of $\Sl_\infty$. We show that $\Sl_\infty$ acts densely on $V^h$. For any $\varphi\in \Glm_\infty$ and $r_1,\cdots, r_n\in V$, we can find an element $\eta\in\Sl_\infty$ such that $\eta(r_k) = h(\varphi)(r_k)$ for $k = 1,\cdots,n$. Therefore $\psi: = h^{-1}(\eta)$ is an element of $\Sl_\infty$ such that $\psi \cdot r_k = \varphi \cdot r_k$ for $k = 1,\cdots, n$. By Proposition \ref{actsdensely}, we conclude that $V^h\in \mathbb{T}_{\Glm_\infty}$.  We know from Corollary \ref{slundef} that either $V^h\cong V$ or $V^h\cong V_*$ as $\Sl_\infty$-modules, therefore we conclude that either $V^h\cong V$ or $V^h\cong V_*$ as $\Glm_\infty$-modules by Lemma \ref{slglm}.
\end{proof}

\begin{mythm}\label{mainglm} Let $h\in \Aut(\Glm_\infty)$.
\begin{enumerate}
\item If there is an isomorphism  $V^h\cong V$ of $\Glm_\infty$-modules, then $h \in \tilde{G}_0$.
\item If there is an isomorphism $V^h\cong V_*$ of $\Glm_\infty$-modules, then $\tau\circ h \in \tilde{G}_0$.
\end{enumerate}
\end{mythm}
\begin{proof}\
\begin{enumerate}
\item If $f\colon V^h\to V$ is an $\Glm_\infty$-isomorphism, then $fh(\varphi) = \varphi f$ for every $\varphi\in\Glm_\infty$. Since $f$ is an isomorphism, we conclude that $h(\varphi) = f^{-1}\varphi f$. What remains to be shown is that $f\in \tilde{G}$. 

Observe that the dual operator $f^*\colon V^*\to (V^h)^*$ is a $\Glm_\infty$-isomorphism. Since $\soc V^* = V_*$ and $\soc (V^h)^*= V_*^h$ by Proposition \ref{socle}, we conclude that $V_*$ is $f^*$-stable. Therefore, $f\in \tilde{G}$ and $h\in \tilde{G}_0$.
\item If $V^h \cong V_*$, then we have the following isomorphisms of $\Glm_\infty$-modules
\[
V^{\tau\circ h}  = (V^h)^\tau\cong V_*^\tau\cong (V^\tau)^\tau = V^{\tau^2} =V.\]
 By statement 1 we conclude that $\tau\circ h\in\tilde{G}_0$.\qedhere
 \end{enumerate}
\end{proof}

\begin{mycol}\label{aut}\
\begin{enumerate}
\item $\Aut(\Glm_\infty) = \tilde{G}_0\rtimes \{\Id, \tau\}$.
\item Every automorphism of $\Sl_\infty$ extends uniquely to an automorphism of $\Glm_\infty$.
\end{enumerate}
\end{mycol}
\begin{proof}
\
\begin{enumerate}
\item Follows directly from Corollary \ref{undef} and Theorem \ref{mainglm}. 
\item Follows directly from statement 1 and Theorem \ref{autsl}.\qedhere\end{enumerate}
\end{proof}

\bigskip
In conclusion, we computed the automorphism group of $\Glm_\infty$ using knowledge about the categories $\mathbb{T}_{\Sl_\infty}$ and $\mathbb{T}_{\Glm_\infty}$. As a problem for the future, it would be interesting to determine the automorphism groups of $\Sl_{U,W}$, $\Gl_{U,W}$ and $\Glm_{U,W}$ for an arbitrary linear system $(U,W)$.

\newpage
\bibliographystyle{alpha}

\begin{thebibliography}{BBCM02}

\bibitem[BBCM02]{BBCM}
K.I. Beidar, M.~Br\"esar, M.A. Chebotar, and W.S. Martindale, 3rd,
\newblock On Herstein's Lie map conjectures, III.
\newblock {\em J. Algebra} 249 (2002), 59-94.

\bibitem[DCPS11]{DPS} 
E.~Dan-Cohen, I.~Penkov, and V.~Serganova,
\newblock A Koszul category of representations of finitary Lie algebras.
\newblock Preprint (2011), arXiv:1105.3407.

\bibitem[Mac43]{Mackey}
G.~Mackey.
\newblock On infinite dimensional linear spaces.
\newblock {\em Trans. AMS} 57 (1945), 155-207.

\bibitem[PS13]{PSer}
I.~Penkov and V.~Serganova.
\newblock Tensor representations of Mackey Lie algebras and their dense
  subalgebras.
\newblock {\em Developments and Retrospectives in Lie Theory: Algebraic Methods}, Springer Verlag (to appear) (2014).

\end{thebibliography}

\end{document}